\documentclass[12pt,a4paper]{amsart}

\usepackage{docmute}
\usepackage{a4wide}
\usepackage[utf8]{inputenc}
\usepackage{amsmath,amssymb}

\usepackage[
backend=biber,
style=numeric,
sorting=nyt,
maxbibnames=9
]{biblatex}
\usepackage{xcolor}
\usepackage{booktabs}
\usepackage{hyperref}
\hypersetup{
	colorlinks,
	citecolor=black,
	filecolor=black,
	linkcolor=blue,
	urlcolor=black
}
\usepackage{stmaryrd}
\usepackage{url}
\usepackage{longtable}
\usepackage[figuresright]{rotating}
\usepackage{amsthm}
\usepackage{bbold}
\usepackage{enumerate}
\usepackage{geometry}
\usepackage[noabbrev]{cleveref}
\newtheorem{definition}{Definition}
\newtheorem{theorem}[definition]{Theorem}
\newtheorem{proposition}[definition]{Proposition}

\newtheorem{lemma}[definition]{Lemma}

\newtheorem{remark}[definition]{Remark}

\newtheorem{corollary}[definition]{Corollary}

\newtheorem{problem}{Problem}

\newtheorem*{claim*}{Claim}

\newcommand{\0}{\emptyset}
\newcommand{\mc}{\mathcal}
\newcommand{\mbb}{\mathbb}

\newcommand{\foralmostall}{\forall^\infty}
\newcommand{\existinfty}{\exists^\infty}
\newcommand{\Ee}{\mc{E}}
\newcommand{\Ii}{\mc{I}}
\newcommand{\Jj}{\mc{J}}
\newcommand{\Kk}{\mc{K}}
\newcommand{\Mm}{\mc{M}}
\newcommand{\Nn}{\mc{N}}

\newcommand{\Dd}{\mc{D}}

\newcommand{\Gg}{\mc{G}}

\newcommand{\Ss}{\mc{S}}

\newcommand{\bez}{\backslash}
\newcommand{\se}{\subseteq}
\newcommand{\sen}{\subsetneq}

\newcommand{\rest}{\hspace{-0.25em}\upharpoonright\hspace{-0.25em}}

\newcommand{\baire}{\omega^\omega}

\newcommand{\concat}{^{\frown}}
\newcommand{\cons}{^{\frown}}

\newcommand{\tn}[1]{\textnormal{#1}}

\newcommand{\cf}{\textnormal{cf}}

\newcommand{\add}{\tn{add}}
\newcommand{\cof}{\tn{cof}}
\newcommand{\non}{\tn{non}}
\newcommand{\cov}{\tn{cov}}

\newcommand{\dom}{\textnormal{dom}}

\newcommand{\Fin}{\textnormal{Fin}}

\def\ca{\mathcal{A}}

\def\cf{\mathcal{F}}
\def\ci{\mathcal{I}}

\def\b{\mathfrak{b}}
\def\c{\mathfrak{c}}
\def\d{\mathfrak{d}}

\def\w{\omega}

\def\baire{\w^\w}

\def\cof{\rm cof}
\def\cov{\rm cov}

\title{An Ideal Zoo in the Baire Space}

\author[Ł. Mazurkiewicz]{Łukasz Mazurkiewicz}
\email{lukasz.mazurkiewicz@pwr.edu.pl}

\author[M. Michalski]{Marcin Michalski}
\email{marcin.k.michalski@pwr.edu.pl}

\author[Sz. Żeberski]{Szymon Żeberski}
\email{szymon.zeberski@pwr.edu.pl}

\thanks{This work has been partially financed by grant {\bf 8211204601, MPK: 9130730000} from the Faculty of Pure and Applied Mathematics, Wrocław University of Science and Technology.
	}

\address{Łukasz Mazurkiewicz, Marcin Michalski, Szymon Żeberski, Faculty of Pure and Applied Mathematics, Wrocław University of Science and Technology, 50-370 Wrocław, Poland}

\date{}

\begin{document}

\begin{abstract}
	In this paper, we study the translations into the Baire space of several well-known $\sigma$-ideals and families originally defined on the Cantor space, using their combinatorial characterizations. These include the ideals of null sets, small sets, those generated by closed measure-zero sets, and the meager sets, leading to their "fake" analogues in the Baire space. We also parametrize families related to null sets by functions from $\baire$. Several structural properties and relations between these families are investigated, including whether they form ideals, the existence of large chains and antichains, orthogonality, the $\kappa$-chain condition, and the determination of certain cardinal invariants.
\end{abstract}

\maketitle

\section{Introduction}

In \cite{MazMiRalZebAddBaire}, the authors considered the ideals $\Mm_-$ and $\Nn$ in the Baire space, designed to parallel the classical ideals of meager and measure-zero sets in the Cantor space. These were obtained by reformulating well-known characterizations in a new setting. Moving beyond the specific function $n \mapsto 2^n$, we will examine a generalization of $\Nn$, denoted $f\Nn(h)$ - the family of fake null sets parameterized by a suitable function $h \in \w^\w$. In addition, we will investigate related families associated with null sets: the small sets $\Ss$ with their parametrized analogues $f\Ss(h)$, and the closed measure zero sets $\Ee$ with $f\Ee(h)$. The study reveals a range of unexpected phenomena and subtle structural distinctions within these generalized frameworks.

\section{Notions and basic observations}

Throughout the paper, we work in ZFC and use standard set-theoretic notation (see, e.g., \cite{Jech}). We denote the set of natural numbers by $\omega$. For sets $X$ and $Y$, we write
    \[
        X^{<\w}=\bigcup_{n\in\w}X^n \quad \tn{ and } \quad Y^X=\{f\se X\times Y:\, f:X\to Y\}.
    \]
 The spaces $2^\w$ and $\w^\w$ will denote the Cantor space and the Baire space respectively. Both are equipped with standard topologies generated by basic clopen sets
    \[
        [\sigma]=\{x\in A^\w:\, \sigma\se x\},
    \]
    where $\sigma\in A^{<\w}$, $A\in\{2,\w\}$.

    We will consider several known families of sets, i.e. the $\sigma$-ideal $\Mm$ of meager subsets of $2^\w$ or $\w^\w$, the $\sigma$-ideal $\Nn$ of null subsets of $2^\w$, the family $\Ss$ of small subsets of $2^\w$ and the $\sigma$-ideal $\Ee$ of subsets of $2^\w$ $\sigma$-generated by closed null sets.

    Let $X$ be an uncountable Polish space and $\Ii\subseteq P(X)$ be a  $\sigma$-ideal. Let us recall some cardinal coefficients:
    \begin{align*}
        \add(\ci)&=\min \{ |\ca|:\; \ca\subseteq \ci\land \bigcup\ca\notin \ci \};
        \\
        \non(\ci)&=\min \{ |A|:\; A\subseteq X\land A\notin \ci\};
        \\
        \cov(\ci)&=\min \{ |\ca|:\; \ca\subseteq \ci\land \bigcup\ca = X\};
        \\
        \cof(\ci)&=\min \{ |\ca|:\; \ca\subseteq \ci\land (\forall A\in \ci)(\exists B\in\ca) (A\subseteq B)\};
        \\
        \b&=\min\{|\cf|: \cf\se\omega^\omega \land (\forall x\in\baire)(\exists f\in\cf)(\existinfty n)(x(n)<f(n))\};
        \\
        \d&=\min\{|\cf|: \cf\se\omega^\omega \land (\forall x\in\baire)(\exists f\in\cf)(\foralmostall n)(x(n)<f(n))\}.
    \end{align*}

The following definition is the rewriting of the characterization of the base of meager sets in the Cantor space (see \cite[Theorem 2.2.4]{BarJu}). It turned out that in the Baire space it forms a $\sigma$-ideal strictly smaller than $\Mm$ (see \cite[Theorem 4]{MazMiRalZebAddBaire}).

\begin{definition}
    $F\in\Mm_-$ if there are $x_F\in\baire$ and a partition of $\w$ into intervals $(I_n)_{n\in\w}$ such that
    $$F\se\left\{x\in\baire:\,(\foralmostall n)(x\rest I_n\ne x_F\rest I_n)\right\}.$$
\end{definition}

In \cite{MazMiRalZebAddBaire}, the authors introduced the notion of fake null sets. We will deal with its natural generalization with respect to any reasonable function from $\w^\w$.

\begin{definition}
    Let $h\in\baire$, $\limsup_n h(n)=\infty$. We will say that $F\in f\Nn(h)$ if
    $$\left(\forall \varepsilon>0\right)\left(\exists(\sigma_n)_{n\in\w}\right)\left(\sum_{n\in\w}\frac{1}{h(|\sigma_n|)}<\varepsilon\land F\se\bigcup_{n\in\w}[\sigma_n]\right).$$
\end{definition}

Notice that $f\Nn(h)$ is a $\sigma$-ideal with $\Gg_\delta$ base.

Let us highlight the fact that the characterization based on \cite[Lemma 2.5.1]{BarJu} and proved for the case of $h(n)=2^n$ in \cite[Lemma 15]{MazMiRalZebAddBaire} also works in the generalized setting.
\begin{lemma}
    Let $h\in\baire$, $\limsup_n h(n)=\infty$. The following are equivalent:
    \begin{enumerate}
        \item $F\in f\Nn(h)$.
        \item There is a sequence $(S_n)_{n\in\w}$, $S_n\se\w^n$, $\displaystyle \sum_{n\in\w}\frac{|S_n|}{h(n)}<\infty$, such that
        $$F\se\left\{x\in\baire:\,(\existinfty n)(x\rest n\in S_n)\right\}.$$
    \end{enumerate}
\end{lemma}

There is a notion related to classical null sets, so-called small sets (explored in e.g. \cite{BarSheSmall}). We generalize this notion too.
\begin{definition}
    Let $h\in\baire$, $\limsup_n h(n)=\infty$. We will say that $F\in f\Ss(h)$ if there is a partition of $\w$ into intervals $(I_n)_{n\in\w}$ and a sequence $(J_n)_{n\in\w}$, $J_n\se\w^{I_n}$, $\sum\frac{|J_n|}{h(|I_n|)}<\infty$ such that
    $$F\se\left\{x\in\baire:\,(\existinfty n)(x\rest I_n\in J_n)\right\}.$$
\end{definition}

In the Cantor space the family of small sets does not form an ideal \cite[Theorem 2.5.7, Theorem 2.5.11]{BarJu}. Hence, $f\Ss(n\mapsto 2^n)$ is not an ideal in the Baire space. Let us pose the following problem.
\begin{problem}
    For which functions $h\in\w^\w$ the family $f\Ss(h)$ is an ideal? A $\sigma$-ideal?
\end{problem}

Considering the characterization of $\Ee$ in the Cantor space that emerges from \cite[Lemma 2.6.3]{BarJu} (see also \cite[Lemma 18]{MiRalZebAddCant}) we may define its generalization in the form of fake $\Ee$ sets.
\begin{definition}\label{definition of fake E}
    Let $h\in\baire$, $\limsup_n h(n)=\infty$. We will say that $F\in f\Ee(h)$ if there is a partition of $\w$ into intervals $(I_n)_{n\in\w}$ and a sequence $(J_n)_{n\in\w}$, $J_n\se\w^{I_n}$, $\sum_{n\in\w}\frac{|J_n|}{h(|I_n|)}<\infty$, such that
    $$F\se\left\{x\in\baire:\,(\foralmostall n)(x\rest I_n\in J_n)\right\}.$$
\end{definition}
Clearly, $f\Ee(h)\se f\Ss(h)$.

In \cite[Lemma 19]{MiRalZebAddCant} nice characterization of $\Ee$ in the Cantor space was obtained. It turns out that if $h$ satisfies a certain condition, then the nice characterization works also in the Baire space.

\begin{proposition}
    Let $h\in\w^\w$ satisfy $h(a+b)\ge h(a)h(b)$ for any $a,b\in\w$ and $c\in(0,1)$. Then we can replace the condition $\sum_{n\in\w}\frac{|J_n|}{h(|I_n|)}<\infty$ in \Cref{definition of fake E} with $(\forall n)(\frac{|J_n|}{h(|I_n|)}<c)$.
\end{proposition}
\begin{proof}
    If $\sum_{n\in \w}\frac{|J_n|}{h(|I_n|)}<\infty$ then for $0<c$ there is $N\in\w$ such that $\sum_{n>N}\frac{|J_n|}{h(|I_n|)}<c$. Hence, $\frac{|J_n|}{h(|I_n|)}<c$ for every $n>N$. We may set $J'_n=J_n$ for $n>N$ and $J'_n=\0$ for $n\le N$ to obtain the same basic $f\Ee(h)$ set.

    Conversely, assume that $(\forall n\in\w)(\frac{|J_n|}{h(|I_n|)}<c)$ for some $0<c<1$. Let $n_0=0$, $n_{k}=n_{k-1}+k$ for $k>0$ and set
    \begin{align*}
        I'_{k}&=\bigcup_{j=n_k}^{n_{k+1}-1}I_j,
        \\
        J'_k&=\{x_{n_k}\cup x_{n_k+1}\cup\ldots\cup x_{n_{k+1}-1}:\, (\forall n_k\le i < n_{k+1})(x_i\in J_i)\}.
    \end{align*}
    Then
    \[
        \frac{|J'_k|}{h(|I'_{k}|)}=\frac{\prod_{j=n_k}^{n_{k+1}-1} |J_j|}{h(\sum_{j=n_k}^{n_{k+1}-1}|I_j|)}\le \prod_{j=n_k}^{n_{k+1}-1} \frac{|J_j|}{h(|I_j|)}<c^k.
    \]
    Therefore $\sum_{n\in \w}\frac{|J'_n|}{|h(|I'_n|)|}<\infty$.
\end{proof}

In each of the above definitions we may replace the condition related to $h$ regulating the size of sets of sequences with finiteness of these sets.
\begin{definition}
    We will say that 
    \begin{itemize}
        \item $F\in f\Nn(\Fin)$ if there is $(S_n)_{n\in\w}$, $S_n\se\w^n$, $|S_n|<\w$ such that 
        \[
            F\se \{x\in\baire:\,(\existinfty n)(x\rest n\in S_n)\}.
        \]
        \item $F\in f\Ss(\Fin)$ if there is a partition of $\w$ into intervals $(I_n)_{n\in\w}$ and a sequence $(J_n)_{n\in\w}$, $J_n\se\w^{I_n}$, $|J_n|<\w$, such that
        \[
            F\se\left\{x\in\baire:\,(\existinfty n)(x\rest I_n\in J_n)\right\}.
        \]
        \item $F\in f\Ee(\Fin)$ if there is a partition of $\w$ into intervals $(I_n)_{n\in\w}$ and a sequence $(J_n)_{n\in\w}$, $J_n\se\w^{I_n}$, $|J_n|<\w$, such that
        \[
            F\se\left\{x\in\baire:\,(\foralmostall n)(x\rest I_n\in J_n)\right\}.
        \]
    \end{itemize}
\end{definition}

We may assume that in the case of $f\Ee(\Fin)$ and $f\Ss(\Fin)$ the intervals can be singletons. We have the following expected inclusions.

\begin{proposition}
    $f\Ee(\Fin)=\Kk_{\sigma}\se f\Nn(\Fin)$.
\end{proposition}

We have also unexpected ones - in the Cantor space it is the other way around.
\begin{proposition}
    $f\Nn(\Fin)\se f\Ss(\Fin)$ and both are $\sigma$-ideals. 
\end{proposition}
\begin{proof}
    Let $F\in f\Nn(\Fin)$. Without loss of generality let 
    \[
        F=\{x\in\w^\w:\, (\existinfty n\in\w)(x\rest n)\in S_n\},
    \]
    where $S_n\se \w^n$ are finite. Set $I_n=\{n\}$ and $J_n=\{(n, \sigma(n)):\, \sigma\in S_{n+1}\}$. Then
    \[
        F\se \{x\in\w^\w:\, (\existinfty n\in\w)(x\rest I_n\in J_n)\}\in f\Ss(\Fin).
    \]

    To see that $f\Ss(\Fin)$ is a $\sigma$-ideal, let $(F_k:\, k\in\w)$ be a sequence of sets from $f\Ss(\Fin)$. With each $F_k$ associate finite $J^k_n\se \w^{\{n\}}$ such that
    \[
        F_k\se \{x\in\w^\w:\,(\existinfty n\in\w)((n,x(n)) \in J^k_n)\}.
    \]
    Set $J_n=\bigcup_{k\le n}J^k_n$. Then 
    \[
        \bigcup_{k\in\w}F_k\se \{x\in\w^\w:\, (\existinfty n\in\w)((n,x(n))\in J_n)\}\in f\Ss(\Fin).
    \]
    The proof in the case of $f\Nn(\Fin)$ is analogous.
\end{proof}

\section{Individual analysis of the families}

\begin{proposition}\label{fake null ortogonalny do M}
    Let $h\in \w^\w$, $\limsup_{n}h(n)=\infty$. Then $f\Nn(h)\perp\Mm$.
\end{proposition}
\begin{proof}
    Fix an enumeration of rationals $\{q_n:\, n\in\w\}=\{q\in \w^\w:\, (\foralmostall n)(q(n)=0)\}$.
    Let $(k_n:\, n\in \w)$ be an increasing sequence of natural numbers such that $(h(k_n):\, n\in\w)$ is increasing and ${\sum_{n\in\w}\frac{1}{h(k_n)}<\infty}$. Set 
    \[
        F=\bigcap_{k}\bigcup_{n\ge k}[q_n \rest h(k_n)].
    \]
    Clearly, $F$ is comeager.
    \\
    Furthermore, notice that $F\in f\Nn(h)$. Indeed, let $S_{h(k_n)}=\{q_n\rest h(k_n)\}$ and $S_n=\0$ otherwise. Then $F=\{x\in\w^\w:\, (\existinfty n\in\w)(x\rest n\in S_n)\}\in f\Nn(h)$.
\end{proof}

    For a $\sigma$-ideal $\Ii$ we say that a set $A$ is $\Ii$-positive, if $A\notin \Ii$.
    \\
    Recall that $\Ii$ is $\kappa$-cc if for every family of Borel $\Ii$-positive sets $\{A_\alpha:\, \alpha<\kappa\}$ there are $\alpha\ne \beta$ such that $A_\alpha\cap A_\beta\notin \Ii$. For the moment, we will focus on the $\kappa$-cc property of our ideals.

    To begin, we state a useful lemma regarding $\Mm_-$.

\begin{lemma}\label{m minus pozytywny meager}
    For every meager set $A$, there exists a nowhere dense set $B\notin \Mm_-$ such that $B\cap A=\0$.
\end{lemma}
\begin{proof}
    Without loss of generality we may assume that $A$ is according to \cite[Lemma 5]{MazMiRalZebAddBaire}, i.e. there is $a: \w^{<\w}\to \w^{<\w}$ such that
    \[
        A=\{x\in \w^\w:\, (\foralmostall \sigma\in\w^{<\w})(\sigma \cons a(\sigma)\not\se x)\}.
    \]
    Let $b: \w^{<\w} \to \w^{<\w}$ be such that $b(\sigma)=\mbb{0}\rest (\max\{|a(\sigma\rest k)|:\, k\in \w\}+\max\sigma+1)$. Set
    \[
        B=\{x\in \w^\w:\, (\forall \sigma\in\w^{<\w})(\sigma \cons b(\sigma)\not\se x)\}\bez A.
    \]
    Clearly $B$ is nowhere dense. To see that $B\notin \Mm_-$ let $F\in \Mm_-$ and let us check that $B\bez F\ne\0$. Without losing generality assume that
    \[
        F=\{x\in \w^\w:\, (\foralmostall n\in\w)(x\rest I_n\ne x_F\rest I_n)\}
    \]
    for some partition $\{I_n:\, n\in\w\}$ of $\w$ into intervals and a pattern $x_F\in \w^\w$. We will construct $x\in B\bez F$ inductively.
    \\
    Let $\xi_0=\0$. At step $n+1$ let
    \begin{align*}
        m_n&=\min\{m\in\w:\, \min I_m > |{\xi_n}\cons a(\xi_n)|+1\},
        \\
        i_n&=\max I_{m_n}+1
    \end{align*}
    and set $\xi_{n+1}=({\xi_n}\cons a(\xi_n)\cons {i_n} \cons \mbb{1}\rest (\min I_{m_n}-|{\xi_n}\cons a(\xi_n)|-1)\cup x_F\rest I_{m_n})\cons 1$. Let $x=\bigcup_{n\in\w}\xi_n$. Notice that $x\notin A\cup F$. We will show that $x\in B$. Let $k\in \w$. There is $n\in \w$ such that $k\in \dom (\xi_{n+1})\bez \dom (\xi_n)$. If $k\in \dom( {\xi_n}\cons a(\xi_n))$, then $|b(x\rest k)|\ge |a(\xi_n)|+1$. If $k\notin \dom( {\xi_n}\cons a(\xi_n))$, then $|b(x\rest k)|>i_n>|I_{m_n}|$. Either way, $(x\rest k) \cons b(x\rest k)\not\se x$.
\end{proof}

\begin{theorem}
    $\Mm_-$ is not $\add(\Mm)$-\tn{cc}.
\end{theorem}
\begin{proof}
    We will construct via induction a family $\{M_\alpha:\,\alpha<\add(\Mm)\}$ of pairwise disjoint meager $\Mm_-$-positive Borel sets.
    \\
    For $\beta<\add(\Mm)$ let $M_\beta\in \Mm\bez\Mm_-$ be obtained according to \Cref{m minus pozytywny meager} for ${A=\bigcup_{\alpha<\beta}M_{\alpha}}$.
\end{proof}
We may ask the following question.
\begin{problem}
    Is $\Mm_-$ not $\c$\tn{-cc}?
\end{problem}
To answer this question it would be productive to describe typical $\Mm_-$-positive Borel sets.

\begin{theorem}
    $f\Nn(\Fin)$ is not $\c$\tn{-cc}.
\end{theorem}
\begin{proof}
    For $f\in2^\w$ define
    $$A_f=\{x\in\w^\w:\,(\forall n\in\w)(2\mid x(n)+f(n))\},$$
    i.e. $f(n)$ decides whether $x(n)$ is even or odd for $x\in A_f$. Clearly $A_f\cap A_g=\emptyset$ for $f\ne g$. Suppose that $A_f\in f\Nn(\Fin)$ and let $(S_n)_{n\in\w}$ witness it. Define $y\in\w^\w$ by
    $$y(n)=2\cdot\max\{x(n):\,x\in S_n\}+2-f(n).$$
    $y\in A_f$, however $(\forall n\in\w)(y\rest n\not\in S_n)$.
\end{proof}

The next part of this section will be devoted to exploring cardinal invariants of $\Mm_-$. In order to do that, we will need the following definitions, introduced in \cite{NewelskiRoslanowski} and \cite{SpiPerfect}.
\begin{definition}
    For $f:\w^{<\w}\rightarrow\w$ let
    $$D_f=\{x\in\w^\w:\,(\foralmostall n)(x(n)\ne f(x\rest n))\}.$$
    Denote $\Dd_\w=\{A\se\w^\w:\,A\se D_f\tn{ for some f}\}$.
\end{definition}

\begin{definition}
    For $x\in\w^\w$ define $K_x=\{y\in\w^\w:\,(\foralmostall n)(y(n)\ne x(n))\}$ and denote by $\Ii_{\tn{ioe}}$ the $\sigma-$ideal generated by $K_x$.
\end{definition}

First, let us prove some basic (non-)inclusions between the $\sigma$-ideals.

\begin{proposition}
    $\Ii_{\tn{ioe}}\se\Mm_-$.
\end{proposition}
\begin{proof}
    $K_x=\{y\in\w^\w:\,(\foralmostall n)(x\rest\{n\}\ne y\rest\{n\})\}\in\Mm_-$.
\end{proof}

\begin{proposition}
    $\Mm_-\not\se\Dd_\w$.
\end{proposition}
\begin{proof}
    Define $F=\{x\in\w^\w:\,(\forall n\in\w)(x(2n+1)\ne x(2n))\}$. $F\in\Mm_-$ since 
    $$F\se\{x\in\w^\w:\,(\forall n\in\w)(x\rest\{2n,2n+1\}\ne\mathbb{0}\rest\{2n,2n+1\})\}\in\Mm_-.$$
    However, $F\notin\Dd_\w$ because
    $$(\forall\sigma\in\w^{<\w})(2\mid |\sigma|\implies(\forall n\in\w)(\sigma\concat n\in F)).\qedhere$$
\end{proof}
\begin{corollary}
    Since $\Ii_{\tn{ioe}}\se\Dd_\w$, $\Ii_{\tn{ioe}}\sen\Mm_-$.
\end{corollary}

\begin{proposition}
    $\Dd_\w\not\se\Mm_-$
\end{proposition}
\begin{proof}
    Take $D_f$ for a bijection $f:\w^{<\w}\rightarrow\w$ and follow the proof of \cite[Theorem 4]{MazMiRalZebAddBaire}.
\end{proof}

Connecting the above facts with the following result (which can be found e.g. in \cite[Theorem 2.3]{KhomLag}) we can calculate $\cov(\Mm_-)$ and $\non(\Mm_-)$.
\begin{theorem}[Y. Khomskii \& G. Laguzzi]\
    \begin{enumerate}
        \item $\cov(\Ii_{\tn{ioe}})=\cov(\Dd_\w)=\cov(\Mm)$ and $\non(\Ii_{\tn{ioe}})=\non(\Dd_\w)=\non(\Mm)$.
        \item $\add(\Ii_{\tn{ioe}})=\add(\Dd_\w)=\w_1$ and $\cof(\Ii_{\tn{ioe}})=\cof(\Dd_\w)=2^\w$.
    \end{enumerate}
\end{theorem}

\begin{corollary}
    $\cov(\Mm_-)=\cov(\Mm)$ and $\non(\Mm_-)=\non(\Mm)$.
\end{corollary}

Now let us proceed to cardinal invariants of the fake null ideal.

\begin{proposition}
    $\cof(f\Nn(\Fin))\leq\d$.
\end{proposition}
\begin{proof}
    Let $\{f_\alpha\in\baire:\,\alpha<\d\}$ be a dominating family, i.e. for every $f\in\baire$ there is $\alpha<\d$ such that $f\leq^*f_\alpha$. Define $S^\alpha_n=\{\sigma\in\w^n:\,(\forall i<n)(\sigma(i)\leq f_\alpha(i))\}$ for every $\alpha<\d$, $n\in\w$ and
    $$F_\alpha=\{x\in\baire:\, (\existinfty n\in\w)(x\rest n\in S^\alpha_n)\}\in f\Nn(\Fin).$$
    Suppose that $F\in f\Nn(\Fin)$, which is witnessed by a sequence $(T_n:\, n\in\w)$. Let $t(n)=\max\{\sigma(i):\,\sigma\in T_n\land i<n\}$. Then there is $\alpha<\d$ such that $t\leq^*f_\alpha$. Therefore $T_n\se S^\alpha_n$ for every $n\in\w$, so $F\se F_\alpha$.
\end{proof}
\begin{proposition}
    $\add(f\Nn(\Fin))\geq\b$.
\end{proposition}
\begin{proof}
    Let $\kappa<\b$ and $\{F_\alpha:\,\alpha<\kappa\}\se f\Nn(\Fin)$. For all $\alpha<\kappa$ there is $(S^\alpha_n:\,n\in\w)$, $S^\alpha_n\in\w^n$, such that $|S^\alpha_n|<\w$ and
    $$F_\alpha\se\{x\in\w^\w:\, (\existinfty n\in\w)(x\rest n\in S^\alpha_n)\}.$$
    For $\alpha<\kappa$ define $f_\alpha\in\w^\w$ by
    $$f_\alpha(n)=\max\{x(k):\,x\in S^\alpha_{n+1}\land k\leq n\}.$$
    Since $\kappa<\b$, there is $f\in\w^\w$ satisfying $f_\alpha\leq^*f$ for every $\alpha<\kappa$. Set $T_n=\{\sigma\in\w^n:\,(\forall i<n)(\sigma(i)<f(n))\}$. Then $|T_n|<\w$ for every $n\in\w$ and
    $$(\forall\alpha<\kappa)(\foralmostall n\in\w)(S^\alpha_n\se T_n).$$
    Hence, $F_\alpha\se F=\{x\in\w^\w:\, (\existinfty n\in\w)(x\rest n\in T_n)\}\in f\Nn(\Fin)$.
\end{proof}

\begin{problem}
    Is it consistent that $\add(f\Nn(\Fin))>\b$ and $\cof(f\Nn(\Fin))<\d$?
\end{problem}

\section{Interactions between the families}

Mirroring the situation in $2^\w$ one could expect that $f\Ee(h)\se f\Nn(h)$ for any reasonable function $h$. The following results show that this expectation is sometimes valid, sometimes not.

\begin{proposition}
    Let $h\in \w^\w$ satisfy $h(a+b)\ge h(a)h(b)$ for any $a,b\in\w$. Then $f\Ee(h)\se f\Nn(h)$.
\end{proposition}
\begin{proof}
    Let $A\in f\Ee(h)$ with associated intervals $I_n, n\in \w$ and sets of patterns $J_n, n\in\w$. Set
    \[
        S_n=\{x_0\cup x_1\cup \dots \cup x_k:\, (x_0, x_1, \dots, x_k)\in J_0\times J_1 \times \dots \times J_k\}
    \]
    for $n=\sum_{i=0}^k|I_k|$ and $S_n=\0$ otherwise. Then, for $n=\sum_{i=0}^k|I_k|$,
    \[
        \frac{|S_n|}{h(n)}=\frac{|J_0|\cdot |J_1|\cdot \ldots \cdot|J_k|}{h(\sum_{i=0}^k|I_i|)}\le \prod_{i=0}^{k}  \frac{|J_i|}{h(|I_i|)}\le C \frac{|J_k|}{h(|I_k|)},
    \]
    where $C=\prod\left\{\frac{|J_i|}{h(|I_i|)}:\, \frac{|J_i|}{h(|I_i|)}>1, i\in\w\right\}$. Hence
    \[
        \sum_{n\in\w}\frac{|S_n|}{h(n)}\le C \sum_{n\in\w}\frac{|J_n|}{h(|I_n|)}<\infty.
    \]
    Notice that $A\se\{x\in\w^\w:\, (\existinfty n\in \w)(x\rest n\in S_n)\}\in f\Nn(h)$.
\end{proof}

\begin{proposition}
    $f\Ee(n\mapsto n^2)\not\se f\Nn(p)$ for any polynomial $p$.
\end{proposition}

\begin{proof}
    We will construct a set $A\in f\Ee(n\mapsto n^2)\bez f\Nn(p)$. Let us find suitable $I_k$ and $J_k$, $k\in\w$. Set
    \begin{align*}
        I_k&=\left[\frac{k(k+1)}{2},\frac{(k+1)(k+2)}{2}\right),
        \\
        J_k&=\{\mbb{0}\rest I_k, \mbb{1}\rest I_k\}
    \end{align*}
    and $A=\{x\in\w^\w:\, (\foralmostall n\in\w)(x \rest I_n\in J_n)\}$. Let $B\in f\Nn(p)$ with associated $(S_n:\, n\in\w)$. We will find $x\in A\bez B$. Let $m_k=\frac{1}{2}(\frac{k(k+1)}{2})(\frac{k(k+1)}{2}+1)$. Notice that
    \[
        [m_k, m_{k+1})=\bigcup_{j=\frac{k(k+1)}{2}}^{\frac{(k+1)(k+2)}{2}-1}I_j.
    \]
    Approximate
    \[
        \sum_{j=m_{k+1}+1}^{m_{k+2}}|\{\tau\rest [m_k, m_{k+1}):\, \tau\in S_j\}|\le \sum_{j=1}^{m_{k+2}}|S_j|\le \sum_{j=1}^{m_{k+2}}p(j)\le q(m_{k+2})=r(k),
    \]
    for some polynomials $q, r$. There exists $K\in\w$ such that $r(k)<2^{k+1}$ for $k\ge K$. Hence, for such $k$ there is $\sigma_k\in \{0,1\}^{[m_k, m_{k+1})}$ such that
    \[
        \sigma_k\notin \bigcup_{j=m_{k+1}+1}^{m_{k+2}}\{\tau\rest [m_k, m_{k+1}):\, \tau\in S_j\}
    \] 
    and
    \[
        (\forall j\in\{\tfrac{1}{2} k(k+1), \ldots , \tfrac{1}{2}(k+1)(k+2)-1\})(\sigma_k\rest I_j\in J_j).
    \]
    Set $x=\mbb{0}\rest m_K \cup \bigcup_{k\ge K}\sigma_k$. Clearly, $x\in A\bez B$.
\end{proof}

The next result also contrasts with properties of $\Ee$ in the Cantor space.

\begin{theorem}
    Let $h\in \w^\w$ be increasing. Then $f\Ee(h)$ is not an ideal.
\end{theorem}
\begin{proof}
    We will construct sets $A,B\in f\Ee(h)$ such that $A\cup B\notin f\Ee(h)$.

    Let $a_0=b_0=0$, $b_1=1$, $E(a_0)=h(b_1)$, $a_1=\min\{l:\,\frac{E(a_0)}{h(l-a_0)}<\frac{1}{2}\}$, $E(b_1)=h(a_1)$. Assume that at $n$-th step we have $(a_k:\, k<n), (b_k:\, k<n)$, $(E(a_k):\, k<n-1), (E(b_k):\, k<n)$. Set
    \begin{align*}
        &b_n=\min\left\{l:\, \frac{E(b_{n-1})}{h(l-b_{n-1})}<\frac{1}{2^n}\right\},
        \\
        &E(a_{n-1})=h(b_n),
        \\
        &a_n=\min\left\{l:\, \frac{E(a_{n-1})}{h(l-a_{n-1})}<\frac{1}{2^n}\right\},
        \\
        &E(b_n)=h(a_n).
    \end{align*}
    Set
    \begin{align*}
        A&=\{x\in\w^\w:\, (\foralmostall n\in\w)(x(a_n)<E(a_n)\,\land\,x\rest (a_n, a_{n+1})=\mbb{0}\rest (a_n, a_{n+1}))\},
        \\
        B&=\{x\in\w^\w:\, (\foralmostall n\in\w)(x(b_n)<E(b_n)\,\land\,x\rest (b_n, b_{n+1})=\mbb{0}\rest (b_n, b_{n+1}))\}.
    \end{align*}
    Notice that indeed $A\in f\Ee(h)$, since $\sum_{n\in\w}\frac{E(a_n)}{h(a_{n+1}-a_n)}\le\sum_{n\in\w}\frac{1}{2^{n+1}}<\infty$. The same remains true for $B$.

    Suppose that there is $C\in f\Ee(h)$ such that $A\cup B\se C$. Let
    \[
        C=\{x\in \w^\w:\, (\foralmostall n\in\w)(x\rest I_n\in J_n)\},
    \]
    with associated partition of $\w$ into intervals $\{I_n:\,n\in\w\}$ and sets of allowed patterns $J_n, n\in\w$ satisfying $\sum_{n\in\w}\frac{|J_n|}{h(|I_n|)}<\infty$. Let $E_A=\{a_k:\, k\in\w\}$ and $E_B=\{b_k:\, k\in\w\}$. Assume that $\max \0=-1$. Without loss of generality
    \[
        N=\{n\in\w:\, \max(I_n\cap E_A)>\max (I_n\cap E_B)\}
    \]
    is infinite and $|J_n|<h(|I_n|)$ for $n\in N$. Denote $a_{k_n}=\max (I_n\cap E_A)$ for $n\in N$. Since $E(a_{k_n})=h(b_{k_n+1})\ge h(|I_n|)$, there exists $e_{k_n}<E(a_{k_n})$ such that $(\forall \sigma\in J_n)(\sigma(a_{k_n})\ne e_{k_n})$. Set $x\in\w^\w$ such that $x(a_{k_n})=e_{k_n}$ for $n\in N$ and $x(i)=0$ otherwise. Clearly $x\in A\bez C$.
\end{proof}

\begin{proposition}
    $f\Nn(\Fin)=\bigcup\left\{f\Nn(h):\, h\in\baire, \limsup_n h(n)=\infty\right\}.$
\end{proposition}
\begin{proof}
    Clearly, if $h\in\baire$, $\limsup_n h(n)=\infty$, $f\Nn(h)\se f\Nn(\Fin)$.
    
    Take $A\in f\Nn(\Fin)$ and $(S_n)_{n\in\w}$ witnessing it. Consider $h(n)=|S_n|\cdot2^n$. $A\in f\Nn(h)$.
\end{proof}

However, analogous equality does not hold for $f\Ss(\Fin)$.

\begin{proposition}
    $\bigcup_{\substack{h\in\baire}}f\Ss(h)\sen f\Ss(\Fin).$
\end{proposition}
\begin{proof}
    Set
    \[
        A=\{x\in \w^\w:\, (\existinfty n\in\w)(x(n)=0)\}.
    \]
    Then $A\in f\Ss(\Fin)\bez f\Ss(h)$ for any $h\in \w^\w$.
    
    To see that $A\in f\Ss(\Fin)$, let $I_n=\{n\}$ and $J_n=\{\{(n,0)\}\}$. To check that $A\notin f\Ss(h)$, fix a partition $\{I_n:\, n\in\w\}$ of $\w$ into intervals and $J_n\se \w^{I_n}, n\in\w$, such that $\sum_{n\in\w}\frac{|J_n|}{h(|I_n|)}<\infty$. Let $N=\{n\in\w:\, |I_n|>1\}$ and notice that $N$ is infinite. Let $x\in \w^\w$ be such that $x\rest I_n\notin J_n$ and moreover for $n\in N$ $x(k)=0$ for some $k\in I_n$. Then $x\in A\bez \{y\in\w^\w:\, (\existinfty n\in\w)(y\rest I_n\in J_n)\}$.
\end{proof}

\subsection*{$f\Nn(f)$ versus $f\Nn(g)$} \hfill\\
Now we will focus on chains and antichains (with respect to $\se$) among the ideals $f\Nn(h)$ for reasonable $h\in\w^\w$. Our analysis begins with the following lemma.

\begin{lemma}\label{lemat przekątniowy}
    Let $s,t\in \w^\w$. Assume that
    \[
        (\forall k\in\w)(\exists N\in\w)(\textstyle \sum_{i=k}^N s(i)>\sum_{i=1}^N t(i)).
    \]
    Then there exists $F=\{x\in\w^\w:\, (\existinfty n)(x\rest n\in S_n)\}$ with $|S_n|=s(n)$ such that for any $F'=\{x\in\w^\w:\, (\existinfty n)(x\rest n\in T_n)\}$ with $T_n$ satisfying $(\foralmostall n)(|T_n|\le t(n))$ it is the case that $F\not\se F'$.
\end{lemma}

\begin{proof}
    Without loss of generality we can assume that $s(0)=0$. Let $a_0^0=0$, $S_0=\0$ and $S_1=\{\sigma^1_i:\,1\leq i\leq s(1)\}$ contain $s(1)$ distinct sequences of length $1$. Suppose now that we have already defined $S_0,S_1,\ldots,S_{a^n_{s(n)}}$ and $a^n_0<a^n_1<\ldots<a^n_{s(n)}$ for some $n\geq 0$. Let $a^{n+1}_0=a^n_{s(n)}$ and $a^{n+1}_i$ be the smallest natural number satisfying
    $$\sum_{j=a^{n+1}_{i-1}+1}^{a^{n+1}_i} s(j)>\sum_{j=1}^{a^{n+1}_i} t(j)$$
    for $0<i\leq s(n+1)$. Then, for each $0<i\leq s(n+1)$, $1\leq l\leq a^{n+1}_i-a^{n+1}_{i-1}$, define sets 

    \[
    S_{a^{n+1}_{i-1}+l}=\left\{{\sigma^{n+1}_i}\cons p\cons 00\ldots 0\in\w^{a^{n+1}_{i-1}+l}:\,\sum_{j=1}^{l-1} s(a^{n+1}_{i-1}+j)\leq p<\sum_{j=1}^{l} s(a^{n+1}_{i-1}+j) \right\}.
    \]
    Moreover, fix an enumeration $S_j=\{\sigma^j_i:\,1\leq i\leq s(j)\}$ for every $a^{n+1}_0< j\leq a^{n+1}_{s(n+1)}$.

    Notice that the above construction satisfies a condition
    \begin{align}\label{lemat_przekątniowy:warunek}
        \textstyle(\forall 0\le p< \sum_{j=a_{i-1}^k+1}^{a_i^k}s(j))(\exists \tau \in \bigcup_{j=a_{i-1}^k+1}^{a_i^k}S_j)(\tau={\sigma_i^k}\concat p\concat 0\concat \dots\concat 0)
    \end{align}
    for every $k\in\w$, $i\in\{1,2,\ldots,s(k)\}$. Finally, set
    \[
        F=\{x\in\w^\w:\, (\existinfty n)(x\rest n\in S_n)\}.
    \]

    Now, let $(T_n:\,n\in\w)$, $T_n\se \w^n$, satisfy $(\foralmostall n)(|T_n|\leq t(n))$. Let
        \[
            F'=\{x\in\w^\w:\, (\existinfty n)(x\rest n\in T_n)\}.
        \]
    We will find $x\in F\bez F'$. There is $n_0\in\w$ such that $s(n_0)>t(n_0)$ and $|T_n|\leq t(n)$ for all $n\geq n_0$. Take any $\xi_0=\sigma_{i_0}^{n_0}\in S_{n_0}\bez T_{n_0}$. From condition (\ref{lemat_przekątniowy:warunek}) there are $\sum_{i=a^{n_0}_{i_0-1}+1}^{a^{n_0}_{i_0}}s(i)$ many candidates for $p$ such that $\sigma^{n_0}_{i_0}\se {\sigma^{n_0}_{i_0}}\concat p \concat 0\dots 0\in S_{n_1}$ for some $n_1\in [a^{n_0}_{i_0-1}+1, a^{n_0}_{i_0}]$. Hence, from the pigeonhole principle ($\sum_{i=a^{n_0}_{i_0-1}+1}^{a^{n_0}_{i_0}}s(i)>\sum_{i=n_0}^{a^{n_0}_{i_0}}|T_i|$), there is $\xi_1=\sigma^{n_1}_{i_1}\in S_{n_1}$ such that $\xi_0\se \xi_1$ and $\xi_1(n_0)\ne \tau(n_0)$ for every $\tau\in\bigcup_{k\in [n_0+1, a^{n_0}_{i_0}]}T_k$.

    Suppose now that we have already defined $\xi_0\se\xi_1\se\ldots\se\xi_m$ such that $\xi_m=\sigma_{i_m}^{n_m}\in S_{n_m}$ and $(\forall k>n_0)(\xi_m\rest k\not\in T_k)$. Again, from condition (\ref{lemat_przekątniowy:warunek}), there are $\sum_{i=a^{n_m}_{i_m-1}+1}^{a^{n_m}_{i_m}}s(i)$ many candidates for $p$ such that $\sigma^{n_{m}}_{i_{m}}\se {\sigma^{n_m}_{i_m}}\concat p \concat 0\dots 0\in S_{n_{m+1}}$ for some $n_{m+1}\in [a^{n_m}_{i_m-1}+1, a^{n_m}_{i_m}]$. Hence, from the pigeonhole principle ($\sum_{i=a^{n_m}_{i_m-1}+1}^{a^{n_m}_{i_m}}s(i)>\sum_{i=n_0}^{a^{n_m}_{i_m}}|T_i|$), we can choose $\xi_{m+1}=\sigma^{n_{m+1}}_{i_{m+1}}\in S_{n_{m+1}}$ such that $\xi_m\se\xi_{m+1}$ and $\xi_{m+1}(n_m)\ne \tau(n_m)$ for all $\tau\in\bigcup_{k\in [n_m+1, a^{n_m}_{i_m}]} T_k$.

    Take $x=\bigcup_{n\in\w}\xi_n$. Since $x\rest n_m=\xi_m\in S_{n_m}$ for every $m\in\w$, $x\in F$. Furthermore, $x\not\in F'$ because $x\rest k\not\in T_k$ for every $k>n_0$.
\end{proof}

\begin{corollary}\label{wniosek po przekątniowym}
    There exists $F=\{x\in\w^\w:\, (\existinfty n)(x\rest n\in S_n)\}$ with $|S_n|=s(n)$ such that for any $F'$ associated with $T_n$ satisfying $(\foralmostall n)(|T_n|<s(n))$, it holds that $F\not\se F'$.
\end{corollary}
\begin{proof}
    Take $t(n)=s(n)-1$. We will check that such $t$ satisfies the assumptions of \Cref{lemat przekątniowy}. Take any $k\in\w$.
    \[
        \sum_{i=1}^N t(i)-\sum_{i=k}^N s(i)=\sum_{i=1}^N (s(i)-1)-\sum_{i=k}^N s(i)=\sum_{i=1}^ks(i)-N,
    \]
    so it suffices to take $N=\sum_{i=1}^k s(i)+1$.
\end{proof}

\begin{corollary}\label{poróżnianie 1}
    Suppose that $\sum_{n\in\w}\frac{f(n)}{g(n)}<\infty$. Then there is $F\in f\Nn(g)\bez f\Nn(f)$.
\end{corollary}
\begin{proof}
    Let $F=\{x\in\w^\w:\, (\existinfty n\in\w)(x\rest n\in S_n)\}$ be as in \Cref{wniosek po przekątniowym}, with $|S_n|=f(n)$. Clearly, $F\in f\Nn(g)$. Assume that $F\in f\Nn(f)$ and let $(T_n)_{n\in\w}$ witness it. Since $\sum_{n\in\w}\frac{|T_n|}{f(n)}<\infty$, $\lim_{n\to\infty}\frac{|T_n|}{f(n)}=0$. Hence, $|T_n|<f(n)$ for almost all $n\in\w$, a contradiction with \Cref{wniosek po przekątniowym}.
\end{proof}

The above corollary implies e.g. that $f\Nn(n\mapsto n^3)\bez f\Nn(n\mapsto n)\ne\0$. However, we still do not know whether $f\Nn(n\mapsto n^2)$ is different from $f\Nn(n\mapsto n)$. A stronger criterion is given in the next theorem.

\begin{theorem}
    Let $\lim_{n\to\infty}\frac{f(n)}{g(n)}=0$. Then $f\Nn(f)\sen f\Nn(g)$.
\end{theorem}
\begin{proof}
    Let $(n_k)_{k\in\w}$ be given by $n_0=0$ and $n_{k+1}=\min\{j\in\w:\, j>n_k\,\land\,g(j)>g(n_k)\}$. See that
    \[
        \lim_{k\to\infty}\frac{\max\{f(j):\, j\le n_k\}}{g(n_k)}=0.
    \]
    Indeed, take $\varepsilon>0$ and let $N_\varepsilon\in\w$ be such that $\frac{f(n)}{g(n)}<\varepsilon$ for every $n\geq N_\varepsilon$. Let $K_\varepsilon\in\w$ such that $n_{K_\varepsilon}\ge N_\varepsilon$ and $\max\{f(j):\,j<N_\varepsilon\}<f(j')$ for some $j'\le n_{K_\varepsilon}$. Then for every $k>K_\varepsilon$
    \[
        \frac{\max\{f(j):\, j\le n_k\}}{g(n_k)}= \frac{f(j')}{g(n_k)}\le \frac{f(j')}{g(j')}<\varepsilon,
    \]
    where $j'\le n_k$ such that $f(j')=\max\{f(j):\, j\le n_k\}$.
    
    Without loss of generality, we may assume that $\frac{\max\{f(j):\, j\le n_k\}}{g(n_k)}<\frac{1}{k2^k}$ for all $k\in\w$. Set
    \[
        h(m)=\begin{cases}
                k\max\{f(j):\, j\le n_k\}, \hfill \tn{ if } (\exists k)(m=n_k);
                \\
                0, \hfill \tn{otherwise}.
            \end{cases}
    \]
    Clearly, $\sum_{n\in \w}\frac{h(n)}{f(n)}=\infty$ and $\sum_{n\in \w}\frac{h(n)}{g(n)}<\infty$.

    Let
    \[
        F=\{x\in\w^\w:\, (\existinfty n\in\w)(x\rest n\in H(n))\},
    \]
    according to Lemma \ref{lemat przekątniowy} with $|H(n)|=h(n)$. Notice that $F\in f\Nn(g)$. Let $F'\in f\Nn(f)$ with associated sequence $(T_n: \, n\in\w)$, $|T_n|=t(n)$ for every natural $n$. Then
    \[
        \sum_{n\in\w}\frac{t(n)}{f(n)}=s<\infty.
    \]
    We will show that the assumptions for Lemma \ref{lemat przekątniowy} are met. Let $j\in\w$ and set $N=n_k$ for $k\ge s$ and $n_k\ge j$. Then
    \[
        \sum_{n=1}^{N}t(n)\le s\cdot\max\{f(i):\, i\le N\} = s\cdot\max\{f(i):\, i\le n_k\}\le h(n_k)=h(N)\le \sum_{n=j}^{N}h(n).
    \]
    Hence $F\not\se F'$.
\end{proof}

    The above theorem implies the existence of chains of these $\sigma$-ideals of length $\c$.

\begin{corollary}
    There is a set $\{f_\alpha\in \w^\w:\,\alpha<\c\}$ such that $f\Nn(f_\alpha)\sen f\Nn(f_\beta)$ or $f\Nn(f_\beta)\sen f\Nn(f_\alpha)$ for $\alpha\ne\beta$.
\end{corollary}
\begin{proof}
    $f_\alpha(n)=\lfloor n^\alpha\rfloor$ for $\alpha\in (1,\infty)$.
\end{proof}
It turns out that there are also antichains of cardinality $\c$.
\begin{theorem}
    There are $f_\alpha\in \w^\w, \alpha < \c$, such that $f\Nn(f_\alpha)\not\se f\Nn(f_\beta)$ for $\alpha\ne\beta$.
\end{theorem}
\begin{proof}
    Let $\{A_\alpha:\, \alpha<\c\}\se P(\w)$ be a m.a.d. family. For each $\alpha<\c$ set
    \[
        f_\alpha(n)=\begin{cases}
	               1, &\tn{ if $n\notin A_{\alpha}$};
	               \\
	               (3n)!, &\tn{ if $n\in A_{\alpha}$}.
	              \end{cases}
    \]
    Fix $\alpha\ne\beta$. Let
    \[
        s(n)=\begin{cases}
	               0, &\tn{ if $n\notin A_{\alpha}$},
	               \\
	               (3n-2)!, &\tn{ if $n\in A_{\alpha}$},
	              \end{cases}
    \]
    and
    \[
        t(n)=\begin{cases}
	               0, &\tn{ if $n\notin A_{\beta}$},
	               \\
	               (3n)!, &\tn{ if $n\in A_{\beta}$}.
	              \end{cases}
    \]
    Clearly $\sum_{n=1}^{\infty}\frac{s(n)}{f_{\alpha}(n)}<\infty$.
    Functions $s$ and $t$ satisfy the assumptions of Lemma \ref{lemat przekątniowy}. Indeed, if $n\in A_\alpha\bez A_\beta$, then 
    \[
        \sum_{i=1}^n t(i)=\sum_{i=1}^{n-1} t(i)\le \sum_{i=1}^{n-1} (3i)!<(3n-3)!\cdot (n-1)<(3n-2)!=s(n).
    \]
    Let $F$ associated with $S_n, n>0,$ as in the thesis of Lemma \ref{lemat przekątniowy}. Clearly, $F\in f\Nn(f_{\alpha})$. Let $F'\in f\Nn(f_{\beta})$ associated with $T_n, n>0,$ i.e. $\sum_{n=1}^\infty\frac{|T_n|}{f_{\beta}(n)}<\infty$. Then there is $N\in\w$ such that $|T_n|\le t(n)$ for $n>N$. By Lemma \ref{lemat przekątniowy} $F\not\se F'$.
\end{proof}

\subsection*{$f\Nn(h)$ versus $f\Ss(h)$} \hfill\\
In this part we will investigate the relationship between $f\Nn(h)$ and $f\Ss(h)$.
Recall that in the Cantor space $\Ss\sen\Nn$ and every null set is a union of two small sets (e.g. \cite{BarSheSmall}).

\begin{proposition}\label{fake null w fake small}
    For every $h\in \w^\w$, $\limsup_{n\to\infty}{h(n)}=\infty$, there exists $h'\in\w^\w$ such that $f\Nn(h)\se f\Ss(h')$.
\end{proposition}

\begin{proof}
    For every $n\in \w$ set $I_{n}=[\frac{n(n+1)}{2},\frac{(n+1)(n+2)}{2})$, i.e. consecutive intervals with lengths increasing by $1$. Set
    \[
        h'(n+1)=\max\{h(k):\, k\in I_{n+1}\}.
    \]
    Take any $F=\{x\in \w^\w:\, (\existinfty n)(x\rest n\in S_n)\}\in f\Nn(h)$. Let
    \[
        J_n=\bigcup_{k\in I_{n+1}}\{\sigma\rest I_n:\, \sigma\in S_k\}.
    \]
    Then
    \[
        |J_n|\le \sum_{k\in I_{n+1}}|S_k|=\sum^{\frac{(n+2)(n+3)}{2}-1}_{k=\frac{(n+1)(n+2)}{2}}|S_k|.
    \]
    Notice that
    \[
        \frac{|J_n|}{h'(n+1)}\le \sum_{k\in I_{n+1}}\frac{|S_k|}{h(k)}.
    \]
    Hence, $\sum_{n\in\w}\frac{|J_n|}{h'(n+1)}\le \sum_{n\in\w}\frac{|S_n|}{h(n)}<\infty$. Therefore, $F'=\{x\in\w^\w:\, (\existinfty n)(x\rest I_n\in J_n)\}$ belongs to $f\Ss(h')$. Since $F\se F'$, $F\in f\Ss(h').$
\end{proof}

\begin{remark}
    There is $h\in\w^\w$ for which $f\Nn(h)\se f\Ss(h)$.
\end{remark}
\begin{proof}
    Take $h(n)=\log n$. Then 
    $$h'(n)=\max\{h(k):\, k\in I_{n}\}\le \log\frac{(n+1)(n+2)}{2}\le 2\log n$$
    for $n\ge 2$.
\end{proof}

\begin{remark}
    $f\Nn(2^n)\not\se f\Ss(2^n)$.
\end{remark}
\begin{proof}
    Take a null set that is not small from $2^\w$ and see it in $\w^\w$.
\end{proof}

\begin{proposition}\label{fake null pasujący do fake smalla}
    For every $h\in \w^\w$, $\limsup_{n\to\infty}{h(n)}=\infty$, there exists $h'\in\w^\w$, $\limsup_nh'(n)=\infty$ such that $f\Nn(h')\se f\Ss(h)$.
\end{proposition}
\begin{proof}
    The proof is largely similar to the one of \Cref{fake null w fake small}, but it is difficult to avoid some repetition without compromising the clarity of the exposition.
    \\
    For every $n\in \w$ set $I_{n}=[\frac{n(n+1)}{2},\frac{(n+1)(n+2)}{2})$. Set $h'(k)=h(n)$ for $k\in I_n$. Clearly, $\limsup_nh'(n)=\infty$.
    \\
    Take any $F=\{x\in \w^\w:\, (\existinfty n)(x\rest n\in S_n)\}\in f\Nn(h')$. Let
    \[
        J_n=\bigcup_{k\in I_{n+1}}\{\sigma\rest I_n:\, \sigma\in S_k\}.
    \]
    Then
    \[
        |J_n|\le \sum_{k\in I_{n+1}}|S_k|=\sum^{\frac{(n+2)(n+3)}{2}-1}_{k=\frac{(n+1)(n+2)}{2}}|S_k|.
    \]
    Notice that
    \[
        \frac{|J_n|}{h(n+1)}\le \sum_{k\in I_{n+1}}\frac{|S_k|}{h(n+1)}=\sum_{k\in I_{n+1}}\frac{|S_k|}{h'(k)}.
    \]
    Hence, $\sum_{n\in\w}\frac{|J_n|}{h(n+1)}\le \sum_{n\in\w}\frac{|S_n|}{h'(n)}<\infty$. Therefore, $F'=\{x\in\w^\w:\, (\existinfty n)(x\rest I_n\in J_n)\}$ belongs to $f\Ss(h)$. Since $F\se F'$, $F\in f\Ss(h).$
\end{proof}

\begin{proposition}
    $f\Ss(h)\not\se f\Nn(\Fin)$ for every $h\in\baire$, $\limsup_n h(n)=\infty$.
\end{proposition}
\begin{proof}
    Since $\limsup_n h(n)=\infty$, there is an increasing sequence $(k_n)_{n\in\w}$ of natural numbers such that
    $$\sum_{n\in\w}\frac{1}{h(k_n)}<\infty.$$
    Take a partition of $\w$ into intervals $(I_n)_{n\in\w}$ such that $|I_{2n}|=1$ and $|I_{2n+1}|=k_n$ for every $n\in\w$. Let $J_{2n}=\emptyset$ and $J_{2n+1}=\{0\}^{I_{2n+1}}$. Clearly,
    $$A=\left\{x\in\baire:\,(\forall n\in\w)(x\rest I_{2n+1}\in J_{2n+1})\right\}\in f\Ss(h).$$

    Take any $(S_k)_{k\in\w}$, $S_k\se\w^k$, $|S_k|<\w$. We will inductively construct $y\in A$ such that
    $$(\forall k\in\w)(y\rest k\notin S_k).$$
    First, take $y(0)=\max\{t(0):\,t\in S_j\land j\in I_0\cup I_1\}+1$ and $y(i)=0$ for every $i\in I_1$. Suppose that we have already defined $y\rest(I_0\cup I_1\cup\ldots\cup I_{2n-1})$. $I_{2n}=\{a\}$, so put $y(a)=\max\{t(a):\,t\in S_j\land j\in I_{2n}\cup I_{2n+1}\}+1$ and $y(i)=0$ for every $i\in I_{2n+1}$.
\end{proof}

    In $2^\w$ every null set is a union of two small sets (see \cite[Theorem 2.5.7, Theorem 2.5.11]{BarJu}). One may wonder is it the case in the Baire space.
    \begin{problem}
        Let $A\in f\Nn(h)$. Are there sets $S_0, S_1\in f\Ss(h)$ such that $A\se S_0\cup S_1$?
    \end{problem}

\subsection*{Orthogonality to $\Mm_-$}\hfill\\
We already proved in \Cref{fake null ortogonalny do M} that $f\Nn(h)\perp\Mm$ for every $h\in\baire$, $\limsup_nh(n)=\infty$. From \Cref{fake null pasujący do fake smalla} it follows that the same holds for $f\Ss(h)$.  Let us examine what happens in the case of $\Mm_-$.
\begin{theorem}
    $f\Ss(h)\perp \Mm_-$ for every $h\in\baire$, $\limsup_n h(n)=\infty$.
\end{theorem}
\begin{proof}
    Since $\limsup_n h(n)=\infty$, there is an increasing sequence $(k_n)_{n\in\w}$ of natural numbers such that
    $$\sum_{n\in\w}\frac{1}{h(k_n)}<\infty.$$
    Take a partition of $\w$ into intervals $(I_n)_{n\in\w}$ such that $|I_n|=k_n$. Set $J_n=\{\mbb{0}\rest I_n\}$.
    Clearly
    \[
    A=\{x\in\baire :\ \exists^\infty n\; x\rest I_n\in J_n\}\in f\Ss(h).
    \]
    Moreover,
    \[
    \baire\bez A=\{x\in\baire :\ \forall^\infty n\; x\rest I_n\neq \mbb{0}\rest I_n\}\in \Mm_-.\qedhere
    \]
\end{proof}

\begin{theorem}
    $f\Nn(Fin)\not\perp \Mm_-$.
\end{theorem}
\begin{proof}
    We will show that for any $F\in\Mm_-$ and $A\in f\Nn(Fin)$,  $F^c\not\se A$. Without loss of generality let $x_F=\mathbb{0}$ and $(I_n)_{n\in\w}$ be a partition of $\w$ into intervals associated with $F$. Let $S_n\se \w^n$, $|S_n|<\w$, be associated with $A$. We will define $y\in F^c\bez A$.
    \\
    Denote $I_{2m+1}=[a_m+1,b_m]$. Set
    \[
        y(k)=
            \begin{cases}
                0, \hfill \tn{ if } a_m<k\le b_m;
                \\
                \max\{\sigma(k):\, \sigma\in S_{k+1}\}+1, \hfill \tn{ \;\;if } b_{m-1}+1\le k < a_m;
                \\
                \max\{\sigma(k):\, \sigma\in\bigcup_{i=a_m+1}^{b_m+1}S_i\}+1, \hfill \tn{ if } k=a_m.
            \end{cases}
    \]
    Notice that $y\rest I_{2m+1}=x_F\rest I_{2m+1}$ for every $m\in\w$, hence $y\in F^c$. On the other hand $(\forall n>0)(y\rest n \notin S_n)$, therefore $y\notin A$.
\end{proof}

\printbibliography

\end{document}